\documentclass[a4paper,11pt, leqno]{amsart}
\usepackage{amssymb}
\usepackage{amsthm}
\usepackage{mathrsfs, amsfonts, amsmath}
\usepackage{xcolor}
\usepackage{graphicx}
\usepackage{tikz}

\usepackage[alphabetic,initials,nobysame]{amsrefs}

\newcommand{\modu}{\operatorname{mod}}

\newcommand{\hatc}{\hat{\mathbb{C}}}

\newtheorem{theorem}{\textbf{THEOREM}}[section]
\newtheorem{proposition}[theorem]{\textsc{Proposition}}
\theoremstyle{definition}
\newtheorem{question}[theorem]{\textsc{Question}}
{\theoremstyle{remark} }
\def\charfn_#1{{\raise1.2pt\hbox{$\chi_{\kern-1pt\lower3pt\hbox{{$\scriptstyle#1$}}}$}}}
\def\leq{\leqslant }
\def\geq{\geqslant }

\def\XXint#1#2#3{{\setbox0=\hbox{$#1{#2#3}{\int}$}
\vcenter{\hbox{$#2#3$}}\kern-.5\wd0}}

\begin{document}

\title{Extremal length and duality}
\author{Kai Rajala} 
	\address{Department of Mathematics and Statistics, University of Jyv\"askyl\"a, P.O. Box 35 (MaD), FI-40014, University of Jyv\"askyl\"a, Finland.}
	\email{kai.i.rajala@jyu.fi}
\date{}

\thanks{  
\newline Research supported by the Research Council of Finland, project number 360505. }
	
\begin{abstract} 
Classical extremal length (or conformal modulus) is a conformal invariant involving families of paths on the Riemann sphere. In ``Extremal length and functional completion'' (\cite{Fug57}), 
Fuglede initiated an abstract theory of extremal length which has since been widely applied. Concentrating on duality properties and applications to quasiconformal analysis, we demonstrate the flexibility of the theory 
and present recent advances in three different settings: 
\begin{enumerate}
\item Extremal length and uniformization of metric surfaces. 
\item Extremal length of families of surfaces and quasiconformal maps between $n$-dimensional spaces.     
\item Schramm's transboundary extremal length and conformal maps between multiply connected plane domains.  
\end{enumerate}
\end{abstract}

\maketitle

\renewcommand{\baselinestretch}{1.2}


\section{Introduction}
Extremal length is an important method in geometric function theory. The formal definition was introduced by Ahlfors and Beurling (\cite{AhlBeu50}), but the method already appeared in earlier works of Ahlfors, Warschawski, Gr\"otzsch, and others, see \cite{Rod74} for an overview. We recall the definition. 

Let $\hatc= \mathbb{C} \cup \{\infty\}$ be the Riemann sphere equipped with the spherical distance and the associated area element $dA$. 
The \emph{conformal modulus} $\modu(\Gamma)$ of a family of paths $\Gamma$ in $\hatc$ is 
\begin{equation}
\label{eq:modudef}  
\modu(\Gamma)=\inf \int_{\hatc} \rho^2 \, dA, 
\end{equation}  
where the infimum is over all Borel functions $\rho:\hatc \to [0,\infty]$ satisfying $\int_\gamma \rho \, ds \geq 1$ for all rectifiable $\gamma \in \Gamma$. 

The extremal length of $\Gamma$ equals $\frac{1}{\modu(\Gamma)}$. We follow modern terminology and talk about (conformal) modulus instead of extremal length. 

Modulus can used to define \emph{quasiconformal maps}: A sense-preserving homeomorphism $f:\Omega \to \Omega'$ between subdomains of $\hatc$ is \emph{$K$-qua\-si\-con\-for\-mal}, $K \geq 1$, if 
\begin{equation}
\label{eq:moduinv} 
\frac{1}{K} \modu(\Gamma) \leq  \modu(f\Gamma)  \leq K \modu(\Gamma) \quad \text{for all path families } \Gamma \text{ in } \Omega. 
\end{equation} 
The \emph{geometric definition} \eqref{eq:moduinv} is equivalent with other familiar definitions of quasiconformality (see \cite{Vai71}), such as the \emph{analytic definition} which requires that $f$ has integrable distributional 
partial derivatives and 
\begin{equation}  
\label{eq:qc} 
||Df||^2 \leq K J_f \quad \text{almost everywhere in } \Omega. 
\end{equation}  
Here $||Df||$ and $J_f$ are the operator norm and the jacobian determinant of $Df$, respectively.   

If $K=1$, then both \eqref{eq:moduinv} and \eqref{eq:qc} are definitions of conformal homeomorphisms $f:\Omega \to \Omega'$. In particular, $\modu$ is a conformal invariant. The conformality of the stereographic projection 
implies that the modulus of a family of paths in $\mathbb{C}$ is not affected if we replace the spherical metric and area element in \eqref{eq:modudef} with the euclidean metric and Lebesgue measure, respectively. 

The basic \emph{Gr\"otzsch identity} concerns modulus on rectangle $R=[0,a]\times[0,b]$: If $\Gamma_1$ and $\Gamma_2$ are the families of paths in $R$ connecting the vertical, respectively horizontal, sides of the boundary, then $\modu(\Gamma_1)=\frac{b}{a}$ and $\modu(\Gamma_2)=\frac{a}{b}$, see \cite{Ahl66}*{pp. 6--8}. In particular, for every $R$ we have 
\begin{equation} 
\label{eq:modudual} 
\modu(\Gamma_1) \cdot \modu(\Gamma_2)=1.  
\end{equation}  
Combining \eqref{eq:modudual} with the Riemann mapping theorem and conformal invariance of modulus, we conclude that \emph{duality property} \eqref{eq:modudual} holds for all \emph{topological rectangles} $R$ in $\hatc$ 
 and path families $\Gamma_1$ and $\Gamma_2$ connecting opposite edges of the boundary of $R$. 

Identity \eqref{eq:modudual} also holds in any topological annulus, or a more general \emph{condenser}, between the family $\Gamma_1$ of paths connecting the two boundary components and the dual family $\Gamma_2$ of paths separating them. Another basic identity concerns the family of paths $\Gamma_{r,R}$ joining the boundary components of concentric annulus $\mathbb{D}(0,R) \setminus \overline{\mathbb{D}(0,r)}$: 
\begin{equation}
\label{eq:concentric} 
\modu(\Gamma_{r,R})=2\pi \Big( \log \frac{R}{r}\Big)^{-1}. 
\end{equation}

Definition \eqref{eq:modudef} extends to path families in $\mathbb{R}^n$ (with the euclidean metric and Lebesgue measure). Moreover, exponent $2$ can be replaced with any $1 \leq p < \infty$, resulting in the 
\emph{$p$-modulus} $\modu_{p}$. The connecting $p$-modulus of a condenser such as a topological annulus equals its \emph{$p$-capacity}, a notion which is widely applied e.g. in potential theory and Sobolev space theory (see \cite{Ric93}*{Proposition II.10.2}).

Conformal invariance of $\modu_{p}$ holds exactly when $p=n$. Modulus estimates are fundamental in \emph{quasiconformal}, \emph{quasiregular} and related mapping theories which started to develop in the 1960s and successfully extend geometric function theory from $\mathbb{C}$ to $\mathbb{R}^n$, see \cite{Vai71}, \cite{Ric93}, \cite{GMP17}.     

Already before such $n$-dimensional considerations, Fuglede \cite{Fug57} presented an abstract theory of modulus of \emph{systems of measures}. In addition to the path families discussed above, such systems include e.g. families of $k$-dimensional surfaces in $\mathbb{R}^n$ as well as path families in general \emph{metric measure spaces}.  The theory (and \emph{Fuglede's lemma} in particular) plays a fundamental role in mapping theory as well as \emph{Analysis in metric spaces}, which includes Sobolev space (\cite{Che99}, \cite{Sha00}, \cite{HKST15}, \cite{ACdM15}, \cite{AILP24}) and potential (\cite{BjoBjo11}) theories based on the \emph{upper gradients} introduced by Heinonen and Koskela (\cite{HeiKos98}). 

We aim to illustrate the flexibility and wide applicability of the modulus method by discussing extensions of duality property \eqref{eq:modudual} in three different settings. Our presentation is motivated by applications to conformal and quasiconformal mappings, and although many of the results below hold for all exponents $1<p<\infty$, we will only consider the conformally invariant cases. Except for the propositions in Section \ref{sec:Transboundary}, which are fairly straightforward consequences of known uniformization theorems, none of the results presented here are new. 

In Section \ref{sec:MS}, we discuss the role of duality in recent developments on the uniformization problem on metric surfaces, and in particular extensions of the characterization of Ahlfors $2$-regular quasispheres by Bonk and Kleiner \cite{BonKle02}. In Section \ref{sec:SF} we discuss the (largely unknown) duality properties of $k$- and $(n-k)$-dimensional surface families in $\mathbb{R}^n$, and potential applications to the quasiconformal parametrization problem along the lines of Semmes \cite{Sem96}, Heinonen and Wu \cite{HeiWu10}, and Pankka and Wu \cite{PanWu14}. Finally, in Section \ref{sec:Transboundary} we discuss the \emph{transboundary modulus} introduced by Schramm \cite{Sch95}, who applied the method to prove important results on conformal uniformization problems such as the long-standing \emph{Koebe conjecture} \cite{Koe08}. 

\section{Extremal length on metric surfaces} \label{sec:MS}

The goal in the \emph{non-smooth uniformization problem} is to extend the classical uniformization theorem to metric spaces under minimal regularity assumptions. Although the discussion below also applies to ``metric surfaces'' with non-trivial topology, we assume for simplicity that $X=(X,d_X)$ is a \emph{metric sphere}, i.e., a metric space homeomorphic to the Riemann sphere $\hatc$ and with finite \emph{two-dimensional Hausdorff measure} $\mathcal{H}^2$. 

Here $\mathcal{H}^2$ is a standard notion of area which can be defined in any metric space, and the requirement $\mathcal{H}^2(X)<\infty$ rules out fractal spheres. It follows from the \emph{coarea inequality} and plane topology that if $V$ is an open subset of a metric sphere $X$ then $\mathcal{H}^2(V)>0$, see e.g. \cite{Raj17}*{Remark 3.4}. 

We seek uniformization by \emph{quasiconformal} and \emph{quasisymmetric maps}. Definition \eqref{eq:modudef} of conformal modulus $\modu$ extends to metric spheres by replacing the spherical metric with $d_X$ and area element 
$dA$ with $d\mathcal{H}^2$ (we normalize $\mathcal{H}^2$ so that $dA=d\mathcal{H}^2$ in $\hatc$). Consequently, we can define quasiconformal homeomorphisms between metric spheres by using the geometric definition \eqref{eq:moduinv}, i.e., by requiring quasi-invariance of $\modu$. 

A homeomorphism $\phi \colon  Y \to Z$ between metric spaces is \emph{quasisymmetric} if there is a homeomorphism 
$\eta:[0,\infty) \to [0,\infty)$ so that 
$$
d_Y(y_1,y_0) \leq t d_Y(y_2,y_0) \quad \text{implies} \quad d_Z(\phi(y_1),\phi(y_0)) \leq \eta(t)d_Z(\phi(y_2),\phi(y_0)). 
$$

In 2002, Bonk and Kleiner (\cite{BonKle02}, see also \cite{Sem91}) gave a characterization for the \emph{Ahlfors $2$-regular} metric spheres admitting quasisymmetric parametrizations from the Riemann sphere. 

\begin{theorem}[\cite{BonKle02}] \label{thm:BonKle}
Let $X$ be an Ahlfors $2$-regular metric sphere. Then there is a quasisymmetric homeomorphism $f: \hatc \to X$ if and only if $X$ is \emph{linearly locally connected}. 
\end{theorem} 
We say that $X$ is \emph{Ahlfors $2$-regular}, if there is $C \geq 1$ so that balls $B(x,r)$ in $X$ satisfy 
\begin{equation} 
\label{ineq:ahlfors}
C^{-1} r^2 \leq \mathcal{H}^2(B(x,r))  \leq Cr^2 \quad \text{for all } x \in X  \text{ and } 0<r< \operatorname{diam}(X). 
\end{equation}
Linear local connectivity is a quantitative connectedness condition which prohibits $X$ from having cusps or ridges. We omit the precise definition. 

By the \emph{local-to-global principle}, the classes of quasiconformal and quasisymmetric homeomorphisms $\phi:Y \to Z$ coincide when both $Y$ and $Z$ have \emph{controlled geometry}, see \cite{HeiKos98}. 
Ahlfors $2$-regular and linearly locally connected metric spheres have controlled geometry, see \cite{BonKle02}. 

In \cite{Raj17} we connected the uniformization problem to duality property \eqref{eq:modudual} by proving that the existence of a quasiconformal 
homeomorphism $f:\hatc \to X$ is equivalent to the following \emph{reciprocality condition}: there is $\kappa \geq 1$ such 
that every topological rectangle in $X$ and path families $\Gamma_1$ and $\Gamma_2$ as in \eqref{eq:modudual} satisfy 
\begin{equation} \label{ineq:reciprocal}
\kappa^{-1} \leq \modu(\Gamma_1) \cdot \modu(\Gamma_2) \leq \kappa. 
\end{equation} 

We also showed that ``upper $2$-regularity'' of $X$, i.e., the second inequality in \eqref{ineq:ahlfors}, implies \eqref{ineq:reciprocal}. Combining our results with the local-to-global principle, we recover Theorem \ref{thm:BonKle} using duality. 

It was proved in \cite{RajRom19} that the lower bound in \eqref{ineq:reciprocal} in fact holds in all metric spheres, and moreover in \cite{EriPog22} that the best possible constant $\kappa^{-1}$ is $\pi^2/16$. Assuming that the upper bound in \eqref{ineq:reciprocal} also holds, i.e., that there is a quasiconformal homeomorphism $f:\hatc \to X$, the best possible quasiconformality constants were established in \cite{Rom19}. Sharpness in each of these results is demonstrated by the plane equipped with the $L_\infty$-norm. 

Metric spheres for which the upper bound in \eqref{ineq:reciprocal} fails for every $\kappa$ can be constructed by collapsing large enough compact subsets $E$ of $\hatc$, such as the unit disk, a segment, or a Cantor set of positive area. More precisely, one chooses a continuous non-negative weight $\sigma$ which vanishes at $E$, and equips $\hatc$ with the associated length (pseudo)metric 
$d_\sigma$ (see \cite{Raj17}, \cite{IkoRom22}). Fine connections between the reciprocality upper bound for such metrics $d_\sigma$ and \emph{conformal removability} of $E$ have recently been studied in \cite{IkoRom22} and \cite{Nta24}. 

Another approach to Theorem \ref{thm:BonKle} was discovered by Lytchak and Wenger (\cite{LytWen20}) as an application of their theory of solutions to \emph{Plateau's problem in metric spaces}. Their 
approach involves \emph{weakly $K$-quasiconformal maps}, i.e., continuous, monotone (i.e., the preimage of every point is connected) 
and surjective mappings $f:\hatc \to X$ satisfying one half of modulus inequality \eqref{eq:moduinv}, namely 
\begin{equation} 
\label{ineq:wqc}
\modu(\Gamma) \leq K \modu(f\Gamma) \quad \text{for all path families } \Gamma \text{ in } \hatc. 
\end{equation} 

Quasiconformal maps $f:\hatc \to X$ are weakly quasiconformal. For the metric spheres $X=(\hatc, d_\sigma)$ above, the projection map $\pi:\hatc \to X$ is always weakly quasiconformal, and (strongly) quasiconformal only if $X$ is reciprocal. Towards an ultimate uniformization result on metric spheres, the author and Wenger conjectured that in fact \emph{every} metric sphere admits a weakly quasiconformal parametrization from $\hatc$. Recently,  Ntalampekos and Romney (\cite{NtaRom24},\cite{NtaRom23}) and Meier and Wenger (\cite{MeiWen24}, assuming $X$ is locally geodesic) proved that this is indeed the case. 

We summarize our discussion in the following theorem which combines the works mentioned above, culminating in \cite{NtaRom24}, \cite{NtaRom23} and \cite{MeiWen24}, and gives a satisfactory answer to the non-smooth uniformization problem in the setting of metric spheres. 

\begin{theorem} \label{thm:MSunif}
Every metric sphere $X$ admits a weakly $\frac{4}{\pi}$-quasiconformal map $f:\hatc \to X$. Moreover, the following properties hold. 
\begin{enumerate} 
\item Constant $\frac{4}{\pi}$ is best possible. 
\item $f$ is a quasiconformal homeomorphism if and only if $X$ satisfies \eqref{ineq:reciprocal} for some $\kappa \geq 1$. 
\item If $X$ satisfies the second inequality in \eqref{ineq:ahlfors}, then $f$ is a quasiconformal homeomorphism. 
\item If $X$ satisfies the second inequality in \eqref{ineq:ahlfors} and is linearly locally connected, then $f$ is a quasisymmetric homeomorphism. 
\end{enumerate} 
\end{theorem} 
 

\section{Extremal length of surfaces families}
\label{sec:SF} 

Let $1 \leq p<\infty$. We say that the \emph{$p$-modulus} $\modu_{k,p}(\Lambda)$ of a family of $k$-dimensional surfaces $\Lambda$ in $\mathbb{R}^n$, $1 \leq k <n$, is 
\begin{equation}
\label{eq:modusurf}  
\inf \int_{\mathbb{R}^n} \rho^p \, dx, 
\end{equation}  
where the infimum is over all Borel functions $\rho:\mathbb{R}^n \to [0,\infty]$ satisfying $\int_\lambda \rho  \geq 1$ for all $\lambda \in \Lambda$. 

In order to complete Definition \eqref{eq:modusurf}, one should define the class of ``$k$-dimensional surfaces'' $\lambda$ and integration over them. When $k=1$, a natural choice is that of (locally) rectifiable paths, which results in a 
straightforward generalization $\modu_p=\modu_{1,p}$ of the classical Definition \eqref{eq:modudef}.   

When $k\geq 2$, the definition of surfaces is more involved because of issues such as regularity and quasiconformal invariance. One may prefer both intrinsic or parametrized surfaces depending on the problem at hand. Below we consider intrinsic surfaces $\lambda$ under varying regularity assumptions. 

The first results concerning $\modu_{k,p}$ were proved by Fuglede \cite{Fug57} for families of Lipschitz $k$-manifolds. The theory of currents \cite{Fed69} can be used towards a general theory, although the invariance issue remains. 

In the 1960s, properties of the conformally invariant $\modu_n$ were developed by Gehring (\cite{Geh61}, \cite{Geh62b}), V\"ais\"al\"a (\cite{Vai61}, \cite{Vai61b}), and Mostow (\cite{Mos68}) to develop a theory of quasiconformal maps in $\mathbb{R}^n$. While surface families played a smaller role in the theory,  (however, see \cite{Rei70}, \cite{Aga71}), they appeared in connection with symmetrization and the following duality result proved by Gehring and Ziemer. 

Let $C(G,C_0,C_1)$ be a \emph{condenser}, i.e., $G$ is a bounded domain in $\mathbb{R}^n$, and $C_0,C_1 \subset \overline{G}$ disjoint compact sets. Denote by $\Gamma$ the family of rectifiable paths $\gamma:[0,1] \to \overline{G}$ so that $\gamma(0) \in C_0$, $\gamma(1) \in C_1$, and $\gamma(t) \in G$ for all $0<t<1$. Moreover, let $\Lambda$ be the family of surfaces \emph{separating $C_0$ and $C_1$ in $C(G,C_0,C_1)$} (see \cite{Zie67} for the precise definition of $\Lambda$). Then 
\begin{equation} 
\label{eq:gehzie}
\Big(\modu_n(\Gamma) \Big)^{\frac{1}{n}}\cdot \Big(\modu_{n-1,\frac{n}{n-1}}(\Lambda) \Big)^{\frac{n-1}{n}}=1. 
\end{equation} 
Gehring (\cite{Geh62}) proved \eqref{eq:gehzie} under a smoothness assumption which Zie\-mer (\cite{Zie67}) was able to remove using the theory of currents. He moreover extended the result to cover general exponents 
(\cite{Zie69}). See \cite{JonLah20} and \cite{LohRaj21} for recent extensions to metric spaces. 

When $1 < k< n-1$, the properties of $\modu_{k,p}$ are not well understood and applications to quasiconformal maps are scarce, even though the closely related differential forms and associated PDEs are fundamental in quasiconformal geometry (see e.g. \cite{DonSul89}, \cite{IwaMar93}, \cite{Iwa92}, \cite{BonHei01}, \cite{Pan08}, \cite{Pry19}, \cite{Kan21b}, \cite{Kan21}, \cite{HeiPan24}).  

Freedman and He (\cite{FreHe91}) studied $\modu_{k,p}$ in solid topological $3$-tori $T$ and showed that duality \eqref{eq:gehzie} does \emph{not} hold in $T$ (see \cite{Loh21} for a generalization to metric tori.) They also considered other versions of $\modu_{k,p}$ and suggested extensions to higher dimensions. 

The most significant applications of surface modulus to quasiconformal maps have been found by Heinonen and Wu (\cite{HeiWu10}) and Pankka and Wu (\cite{PanWu14}), who extended the three-dimensional results of Semmes \cite{Sem96} by showing that for every $n \geq 4$ there are $n$-manifolds $X$ which geometrically resemble $\mathbb{R}^n$ but are not quasiconformally equivalent to it. It follows that Theorem \ref{thm:BonKle} does not extend to any $n \geq 3$. 

Both \cite{HeiWu10} and \cite{PanWu14} apply classical spaces with ``wild topology'', such as the Whitehead manifold, equipped with the ``good'' metrics introduced in \cite{Sem96}. Roughly speaking, they show that specific families of surfaces, which arise from wildness and satisfy quasiconformal quasi-invariance of $\modu_{k,\frac{n}{k}}$, also satisfy duality similar to \eqref{eq:gehzie} in $\mathbb{R}^n$ 
but not in $X$ even if equality is replaced with multiplicative bounds as in \eqref{ineq:reciprocal}. Quasi-invariance then implies that there cannot be quasiconformal maps $f:\mathbb{R}^n \to X$. 

It is an open problem whether the duality results \eqref{eq:modudual} and \eqref{eq:gehzie} hold for $\modu_{k,p}$ when $2 \leq k \leq n-2$. We present a recent result of Lohvansuu (\cite{Loh23}), which proves  one half of duality under mild conditions. 

Let $Q \subset \mathbb{R}^n$ be a compact set and assume that there is a homeomorphism $h:Q \to I^n$ onto the unit $n$-cube $I^n$. Fix $1 \leq k \leq n-1$, and denote 
$$
A=h^{-1}(\partial I^k \times I^{n-k})  \quad \text{and} \quad B=h^{-1}(I^k \times \partial I^{n-k}). 
$$
We assume that $A$, $B$ and $Q$ are locally Lipschitz neighborhood retracts, and denote the \emph{Lipschitz homology groups} with integer coefficients by $H^L_*$ (see \cite{Loh23} for further details). 

Let $\Lambda_A$ be the collection of the images of relative Lipschitz $k$-cycles of $Q \setminus B$ that generate $H^L_k(Q,A)$, and define the modulus of $\Lambda_A$ as in \eqref{eq:modusurf}, where integration over $\lambda \in \Lambda_A$ is with respect to $\mathcal{H}^k$. Similarly, let $\Lambda_B$ be the dual family of $(n-k)$-dimensional surfaces which are the generators of $H^L_{n-k}(Q,B)$. 

\begin{theorem}[\cite{Loh23}]
\label{thm:loh}
\begin{equation} \label{eq:loh}
\Big(\modu_{\frac{n}{k}} (\Lambda_A) \Big)^{\frac{k}{n}} \cdot \Big(\modu_{\frac{n}{n-k}} (\Lambda_B) \Big)^{\frac{n-k}{n}} \leq 1. 
\end{equation} 
\end{theorem} 
The validity of the opposite inequality in Theorem \ref{thm:loh} is not known.  
\begin{question}
Does equality hold in \eqref{eq:loh}? 
\end{question} 
Recently, Kangasniemi and Prywes (\cite{KanPry24}) proved that duality does hold if one replaces the modulus defined in \eqref{eq:modusurf} with a \emph{$p$-differential form modulus} 
of families of Lipschitz surfaces. Moreover, their result extends from $\mathbb{R}^n$ to Lipschitz Riemannian manifolds. Surface modulus in the context of extensions of \emph{Whitney's geometric integration theory} (see \cite{Whi57}, \cite{Hei05}) to metric spaces has been considered in \cite{RajWen13} and \cite{PRW15}. 



\section{Transboundary extremal length} 
\label{sec:Transboundary} 
Transboundary modulus is a generalization of the classical modulus and a powerful tool for solving both classical and modern uniformization problems. 
The method was introduced by Schramm (\cite{Sch95}) who studied uniformization of conformal maps between multiply connected subdomains of $\hatc$. 

The main open problem in the field is 
\emph{Koebe's conjecture} (\cite{Koe08}) which asserts that every such domain admits a conformal map onto a \emph{circle domain} whose collection of complementary components consists of disks and points. The most significant partial result is the following theorem by He and Schramm (\cite{HeSch93}) (Koebe himself proved the finitely connected case). 

\begin{theorem}[\cite{HeSch93}]\label{thm:hesch}
Every countably connected\footnote{i.e., the collection of connected components of $\hatc \setminus G$ is countable} domain $G \subset \hatc$ admits a conformal map onto a circle domain. 
\end{theorem}
See \cite{HeSch93} for further references, \cite{EsmRaj24}, \cite{KarNta24}, \cite{NtaRaj24}, \cite{NtaYou20}, \cite{Raj24} and \cite{You16} for recent results related to the conjecture, and \cite{BonMer13} 
for a notable development of transboundary modulus towards modern uniformization.  

Schramm (\cite{Sch95}) applied transboundary modulus to give a simplified proof for Theorem \ref{thm:hesch} and to establish a \emph{cofat uniformization theorem}: every subdomain 
of $\hatc$ whose complementary components are uniformly Ahlfors $2$-regular satisfies Koebe's conjecture. 

We recall the definition. Given a domain $G \subset \hatc$, we denote by $\mathcal{C}(G)$ the collection of connected components of 
$\hatc \setminus G$ and let $\hat{G}=\hatc / \sim$, where 
$$
x \sim y \text{ if either } x=y \in G \text{ or } x,y \in p \text{ for some } p \in \mathcal{C}(G). 
$$
Let $\pi_G:\hatc \to \hat G$ be the associated projection map. 
  
Fix a Borel function $\rho\colon \hat G \to [0,\infty]$ and a path $\gamma\colon [a,b]\to  \hat G$. Then $\gamma^{-1}( \pi_{G}(G))$ has countably many components $O_j\subset [a,b]$, $j\in J$. 
We denote $\gamma_j= \gamma|_{O_j}$ and $\alpha_j = \pi_{G}^{-1}\circ \gamma_j$, and define
\begin{align*}
\int_{\gamma} \rho \, ds= \sum_{j\in J}\int_{\alpha_j} \rho\circ \pi_{G} \, ds,
\end{align*}
where the integral is understood to be infinite if one of the paths $\alpha_j$ is not locally rectifiable. 

Let $\Gamma$ be a family of paths in $\hat G $.  A Borel function $\rho\colon \hat G \to [0,\infty]$ is \textit{admissible} for $\Gamma$ if 
\begin{align*}
\int_{\gamma} \rho \, ds + \sum_{\substack{p\in \mathcal C(G)\\ |\gamma|\cap p\neq \emptyset}} \rho( p)\geq 1
\end{align*}
for each $\gamma\in \Gamma$. Here $|\gamma|$ is the image of $\gamma$. The \textit{transboundary modulus} of $\Gamma$ with respect to the domain $G$ is 
\begin{align*}
\modu_{G}(\Gamma)=  \inf_{\rho}\Big\{ \int_{G}(\rho\circ \pi_{G})^2 \, dA + \sum_{p\in \mathcal C(G)}\rho(p)^2\Big\},  
\end{align*}
where the infimum is over admissible functions $\rho$. 

It is straightforward to check that conformal invariance (as well as quasi-invariance \eqref{eq:moduinv}) holds for transboundary modulus. For finitely connected domains whose complementary components are suitably round, such as the cofat domains above, $\modu_G$ behaves somewhat like the classical modulus $\modu$; one can prove qualitative estimates similar to the Gr\"otzsch identity and \eqref{eq:concentric} (\cite{Sch95}, \cite{Bon11}, \cite{Raj24}, \cite{NtaRaj24}). 

Such estimates are not valid in general domains, but one could still hope for duality \eqref{eq:modudual} to hold for transboundary modulus. It is not difficult to show that duality holds for finitely connected domains. Countably connected domains are more challenging, but one can rely on Theorem \ref{thm:hesch} to prove duality. 

\begin{proposition} \label{prop:countablemodu}
Let $G \subset \hatc$ be a countably connected domain and let $J_1\cup J_2\cup J_3 \cup J_4 \subset G$ be the boundary of a topological rectangle $R \subset \hatc$ with boundary edges $J_j$ in cyclic order. 
Moreover, let $\Gamma_1$ be the family of paths joining $\pi_G(J_1)$ and $\pi_G(J_3)$ in $\pi_G(R)$, and let $\Gamma_2$ be the family of paths joining $\pi_G(J_2)$ and $\pi_G(J_4)$ in $\pi_G(R)$. 
Then 
$$
\modu_G(\Gamma_1)\cdot \modu_G(\Gamma_2)=1. 
$$
\end{proposition}
We prove Proposition \ref{prop:countablemodu} in Section \ref{sec:proofs}. Duality does not hold without the countability assumption; if $G$ is the complement of a Cantor set of positive area, then an argument involving parallel segments and Fubini's theorem shows that one can choose (geometric) squares $R$ in Proposition \ref{prop:countablemodu} so that $\modu_G(\Gamma_1)\cdot \modu_G(\Gamma_2)$ becomes arbitrarily large. 

Thus the \emph{duality upper bound} does not hold in general. The following extension of Proposition \ref{prop:countablemodu} however shows that the \emph{duality lower bound} holds for all domains 
that satisfy Koebe's conjecture. 

\begin{proposition} 
\label{prop:koebelower}
Suppose domain $G \subset \hatc$ admits a conformal map onto a circle domain. Then 
\begin{equation} 
\label{ineq:lower}
\modu_G(\Gamma_1)\cdot \modu_G(\Gamma_2)\geq 1
\end{equation} 
for all topological rectangles $R$ as in Proposition \ref{prop:countablemodu}. 
\end{proposition} 
We prove Proposition \ref{prop:koebelower} in Section \ref{sec:proofs}. We do not know if the duality lower bound holds in general. 

\begin{question} 
\label{ques:lower}
Does \eqref{ineq:lower} hold for all domains $G \subset \hatc$? 
\end{question} 
By Proposition \ref{prop:koebelower}, a negative answer to Question \ref{ques:lower} would give a negative answer to the Koebe conjecture. 
On the other hand, a positive answer could be considered evidence towards the Koebe conjecture.


\section{Proofs of Propositions \ref{prop:countablemodu} and \ref{prop:koebelower}} \label{sec:proofs}
Let $G \subset \hatc$ be a domain satisfying the assumptions of either of the propositions, and let $R$ be a topological rectangle with edges $J_j \subset G$ as above. 

\begin{proposition}\label{prop:squareunif}
There is a conformal homeomorphism $f:R \cap G \to G' \subset Q=[0,1] \times [0,b]$, $b>0$, with a continuous extension $\tilde{f}$ to $\partial R$, so that 
\begin{itemize} 
\item[(i)] $\tilde{f}$ maps each $J_j$ onto a boundary edge $I_j$ of $Q$, and 
\item[(ii)] the collection of bounded complementary components of $G'$ consists of points and squares whose sides are parallel to coordinate axes. 
\end{itemize} 
\end{proposition} 

\begin{proof}
Recall that countably connected domains satisfy Koebe's conjecture by Theorem \ref{thm:hesch}. In particular, since $G$ is assumed to satisfy the assumptions of one of the propositions, 
we may assume that $G$ is a circle domain. Schramm's cofat uniformization theorem (\cite{Sch95}*{Theorem 4.2)} then shows the existence of a conformal map $f$ from $R \cap G$ onto a domain 
satisfying Condition (ii). We modify Schramm's method to construct $f$ so that also Condition (i) is satisfied, leaving the details to the reader. See \cite{Nta20}, \cite{Raj17} for similar arguments.  

We first assume that $R \cap G$ is finitely connected. There is a unique minimizer $\rho_0:\hat{G} \to [0,\infty]$ of $\modu_G(\Gamma_1)$ for the path family $\Gamma_1$ defined in Proposition \ref{prop:countablemodu}, i.e., $\rho_0$ is admissible for $\Gamma_1$ and  
\begin{equation}
\label{eq:mingrad} 
\int_{G}(\rho_0\circ \pi_{G})^2 \, dA + \sum_{p\in \mathcal C(G)}\rho_0(p)^2 = \modu_G(\Gamma_1). 
\end{equation}
The restriction of $\rho_0$ to $R \cap G$ is the gradient of a harmonic function $u$ which extends continuously to constant $u=0$ on $J_1$ and constant $u=1$ on $J_3$. 

Moreover, $u$ has a harmonic conjugate $v$ which equals $0$ on $J_2$ and $b=\modu_G(\Gamma_1)$ on $J_4$. One can check that 
$$
f=(u,v):R \cap G \to G' \subset [0,1] \times [0,b]
$$ 
is a conformal homeomorphism satisfying Conditions (i) and (ii). 

The above argument works for all finitely connected domains, without conditions on the shapes of the bounded complementary components. See \cite{Bon16} for a similar construction.  

We now assume that $R \cap G$ is infinitely connected. Since $G$ is a circle domain, the collection $\{p_1,p_2,\ldots \}$ of complementary components with positive diameter is countable (or finite). 
Moreover, we may assume that the complement of $G$ equals the closure of the union of sets $p_j$.  

We approximate $G$ with finitely connected domains $G_j=\hatc \setminus \{p_1,\ldots, p_j\}$. Since $G$ is a circle domain, it is also a cofat domain. Thus  Schramm's 
\emph{extended Carath\'eodory kernel convergence theorem} (\cite{Sch95}*{Therorem 3.1}) applied to the conformal maps $f_j:R \cap G_j \to G'_j$ constructed above gives the existence of a 
limit map $f$ satisfying the requirements of the proposition. 
\end{proof}

We denote by $\Lambda_1$ the family of paths joining the projections of the vertical boundary edges of $Q$ in $\pi_{G'}(Q)$, where $Q$ is the rectangle in Proposition \ref{prop:squareunif}, and by 
$\Lambda_2$ the corresponding family of paths joining the projections of the horizontal boundary edges. We claim that  
\begin{equation} 
\label{eq:heihei} 
\modu_{G'}(\Lambda_1) \geq b \quad \text{and} \quad \modu_{G'}(\Lambda_2) \geq \frac{1}{b}. 
\end{equation} 
Both estimates are proved using the same argument, so we only consider $\Lambda_1$. 

By Property (ii) above, the collection of bounded complementary components of $G'$ with positive diameter consists of a finite or countable number of squares $Q_k \subset Q$. We fix  
an admissible function $\rho:\hat {G'} \to [0,\infty]$ for $\Lambda_1$. Without loss of generality, we may assume that 
\begin{equation} 
\label{eq:tsar} 
\sum_{p \in \mathcal{C}(G')} \rho(p)^2 < \infty  
\end{equation} 
(recall that $\mathcal{C}(G')$ is the collection of complementary components of $G'$). Given $0<t<b$ let $\gamma_t = \pi_{G'} \circ I_t \in \Lambda_1$, where $I_t$ is the horizontal segment parametrized by arclength starting at $(0,t)$ and ending at $(1,t)$, and $\pi_{G'}$ the projection map. By admissibility, 
\begin{equation} 
\label{eq:tref}
1 \leq \int_{\gamma_t} \rho \, ds + \sum_{p \in \mathcal{C}(G')} \rho(p) \quad \text{for every } 0<t<b. 
\end{equation} 
Moreover, since $\rho(p)>0$ for at most countably many point components $p \in \mathcal{C}(G')$ by \eqref{eq:tsar}, we conclude from \eqref{eq:tref} that in fact 
\begin{equation} 
\label{eq:tref2}
1 \leq \int_{\gamma_t} \rho \, ds + \sum_k \rho(Q_k) \quad \text{for all but countably many } 0<t<b; 
\end{equation} 
here and in what follows, we identify the points $x \in G'$ and complementary components $p$ of $G'$ with their projections in $\hat{G'}$. Integrating \eqref{eq:tref2} over $t$ and applying Fubini's theorem, we obtain   
$$
b \leq \int_{G'} \rho \, dA + \sum_k \delta_k \rho(Q_k)=I_1, 
$$
where $\delta_k$ is the sidelength of $Q_k$. Applying H\"older's inequality yields  
\begin{eqnarray*}
I_1^2 \leq \Big(|G' |_2 +\sum_k \delta_k^2 \Big)\cdot\Big(\int_{G'} \rho^2 \, dA + \sum_k \rho(Q_k)^2\Big)=I_2 \cdot I_3.  
\end{eqnarray*}
Here $|\cdot|_2$ denotes the Lebesgue measure. Notice that 
$$ 
I_2=|G'|_2 +\sum_k |Q_k|_2 \leq |Q|_2=b, 
$$
so combining the estimates gives 
\begin{equation} \label{eq:tell}
b \leq I_3= \int_{G'} \rho^2 \, dA + \sum_k \rho(Q_k)^2. 
\end{equation} 
Since \eqref{eq:tell} holds for all admissible functions, we conclude the desired bound \eqref{eq:heihei} for $\Lambda_1$. Proposition \ref{prop:koebelower} follows by combining Proposition 
\ref{prop:squareunif}, \eqref{eq:heihei}, and the conformal invariance of transboundary modulus. 

We now assume that $G$ is countably connected. Then domain $G'$ in Proposition \ref{prop:squareunif} is also countably connected. We claim that now also the opposite inequalities to \eqref{eq:heihei} hold, i.e., 
\begin{equation} 
\label{eq:heihei2} 
\modu_{G'}(\Lambda_1) \leq b \quad \text{and} \quad \modu_{G'}(\Lambda_2) \leq \frac{1}{b}. 
\end{equation} 
It is again sufficient to consider $\Lambda_1$. Let $\rho:\hat{G'} \to [0,\infty]$ be the function defined by $\rho(z)=1$ for $z \in G'$, $\rho(Q_k)=\delta_k$ for squares 
$Q_k$ above (recall that $\delta_k$ is the sidelength), and $\rho(p)=0$ elsewhere. Let $\gamma \in \Lambda_1$, and denote by $\pi_1:\mathbb{C} \to \mathbb{R}$ be the projection to the real axis. 
Then the countability assumption yields  
$$
1 \leq |\pi_1(|\gamma| \cap G')|_1+ |\pi_1(\cup \{Q_k \in |\gamma| \}|_1 \leq \int_\gamma \rho \, ds + \sum_{Q_k \in |\gamma|} \delta_k. 
$$ 
We conclude that $\rho$ is admissible for $\Lambda_1$. Since 
$$
\int_{G'} \rho^2 \, dA + \sum_{p \in \mathcal{C}(G')} \rho(p)^2 = |G'|_2 + \sum_k \delta_k^2 =|Q|_2=b, 
$$
we conclude the desired bound \eqref{eq:heihei2} for $\Lambda_1$. Proposition \ref{prop:countablemodu} follows by combining Proposition 
\ref{prop:squareunif}, \eqref{eq:heihei}, \eqref{eq:heihei2}, and the conformal invariance of transboundary modulus. 

\subsection*{Acknowledgment}
We are grateful to Dimitrios Ntalampekos for the discussions on Koebe's conjecture, which inspired Question \ref{ques:lower}.

\vskip 10pt
\noindent

\bibliographystyle{alpha}
	\bibliography{FugledeBiblio}

\end{document}